\documentclass[a4paper,11pt]{article}
\usepackage{amsmath}
\usepackage{amssymb}
\usepackage{amsthm}
\usepackage{graphicx}
\usepackage{amscd}
\usepackage{epic, eepic}
\usepackage{url}
\usepackage{color}
\usepackage[utf8]{inputenc} 
\usepackage{comment}
\usepackage{enumerate}   
\usepackage{graphicx}
\usepackage{epstopdf}
\usepackage{enumitem}
\usepackage{tikz}
\usepackage[top=24mm, bottom=25mm, left=25mm, right=25mm]{geometry}
\usepackage[bookmarks=false,hidelinks]{hyperref}
\usepackage{mathrsfs}
\usetikzlibrary{arrows}
\definecolor{wrwrwr}{rgb}{0.3803921568627451,0.3803921568627451,0.3803921568627451}
\newcommand{\abs}[1]{\left\vert{#1}\right\vert}

\newtheorem{theorem}{Theorem}
\newtheorem{claim}[theorem]{Claim}

\newtheorem{conjecture}{Conjecture}

\def\C{\mathcal C}
\def\P{\mathcal P}
\def\F{\mathcal F}
\newcommand{\ex}{{\rm  ex}}

\usepackage{authblk}

\title{The maximum number of cliques in graphs with bounded odd circumference}
\author[1,3]{Zequn Lv}
\author[1]{Ervin Győri}
\author[1,3]{Zhen He}
\author[1]{Nika Salia}
\author[1,4]{Chuanqi Xiao}
\author[1,2]{Xiutao Zhu}
\date{}

\affil[1]{Alfr\'ed R\'enyi Institute of Mathematics, Hungarian Academy of Sciences. }
\affil[2]{Department of Mathematics, Nanjing University.}
\affil[3]{Department of Mathematical Sciences, Tsinghua University.}
\affil[4]{Department of Computer Science and Information Theory, Budapest University of Technology and Economics.}

\begin{document}

\maketitle
\begin{abstract}
In this work, we give the sharp upper bound for the number of cliques in graphs with bounded odd circumferences.  
This generalized Turán-type  result is an  extension of the celebrated Erd\H{o}s and Gallai theorem and a strengthening of Luo's recent result. 
The same bound for graphs with bounded even circumferences is a trivial application of the theorem of Li and Ning.
\end{abstract}

\section{Introduction}

A central topic of extremal combinatorics  is to investigate sufficient conditions for  the appearance of given substructures.
In particular, for a given graph $H$ and a set of graph $\F$, the generalized Tur\'an number $\ex(n,H,\F)$ denotes the maximum number of copies of $H$ in a graph on $n$ vertices containing no $F$ as a subgraph, for every $F\in \F$. 
In 1959, Erd\H{o}s and Gallai \cite{erdHos1959maximal} determined the maximum number of edges in a graph with a small circumference. 
For every integers $n$ and $k$ such that $n \geq k \geq 3$, they proved 
\begin{equation*}\label{equation:Erdos_gallai_paths}
\ex(n,K_2,\C_{\geq k}) \leq \frac{(k - 1)(n - 1)}{2},    
\end{equation*} 
where $K_k$ denotes the complete graph with $k$ vertices and $\C_{\geq k}$ denotes the family of cycles of length at least $k$. 
The bound is sharp for every $n$ congruent to one modulo $k-2$.  The equality is attained by graphs with  $\frac{n-1}{k-2}$ maximal $2$-connected blocks each isomorphic to $K_{k-1}$. 
Recently Li and Ning~\cite{li2021eigenvalues} proved that in order to obtain the same upper-bound it is enough to forbid only long even cycles 
\begin{equation}\label{Equation:Li_Ning_edges_no_even_cycle}
\ex(n,K_2,\C^{even}_{\geq 2k}) \leq \frac{(2k - 1)(n - 1)}{2},
\end{equation}
where $\C^{even}_{\geq 2k}$ denotes the family of even cycles of length at least $2k$,  that is $\{C_{2k},C_{2k+2},\dots\}$. We also denote the family of odd cycles of length at least $2k+1$ by $\C^{odd}_{\geq 2k+1}:=\{C_{2k+1},C_{2k+3},\dots\}$.

Note that by considering the complete balanced bipartite graph, one can not achieve a linear bound for edges in graphs with a small odd circumference. On the other hand, Voss and Zuluga~\cite{voss1977maximale} proved that every $2$-connected graph $G$ with minimum degree at least   $k\geq3$, with at least $2k+1$ vertices, contains an even cycle of length at least $2k$. Even more, if $G$ is not bipartite then it contains an odd cycle of length at least $2k-1$.

The celebrated Erd\H{o}s and Gallai theorem is well studied and well understood. There are numerous papers strengthening, generalizing, and extending for different classes of graphs. Recently Luo~\cite{luo2018maximum} proved 
\begin{equation}\label{equation:Luo}
\ex(n,K_r,\C_{\geq k})\leq \frac{(n-1)}{k-2}\binom{k-1}{r},    
\end{equation}
for all $n\geq k \geq 4$. This bound is a great tool for obtaining results in hypergraph theory. In order to highlight its significance, that bound was consequently reproved with different methods multiple times~\cite{ning2020extensions, zhu2022stability}. 
In this paper, we strengthen Luo's theorem. In particular, we obtain the tight same bounds for graphs with bounded odd circumferences. 
On the other hand, the result for graphs with bounded even circumference is trivial after applying Equation~\ref{Equation:Li_Ning_edges_no_even_cycle} and the following easy corollary of Equation~\ref{equation:Luo},
\[
\ex(n,K_r,P_{k+1}) \leq \frac{n}{k}\binom{k}{r},    
\]
where $P_{k+1}$ denotes the path of length $k$.
Since the graph $G$ does not contain a cycle of length $2k$, for each vertex the neighborhood contains no path of length $2k-2$.  
In particular, for all $r\geq 3$ we have
\[
\ex(n,K_r,\C^{even}_{\geq 2k})\leq \frac{1}{r}\sum_{v}\frac{d(v)}{(2k-2)}\binom{2k-2}{r-1}\leq \frac{1}{r}\frac{ex(n,K_{2},\C^{even}_{\geq 2k})}{(k-1)}\binom{2k-2}{r-1}
\leq \frac{n-1}{2k-2}\binom{2k-1}{r}.
\]
 For graphs with small odd circumferences, we have the following theorem.

\begin{theorem}\label{thm:K_r_odd-free}
For integers $n,k,r$ satisfying $n\geq 2k \geq r$ and $k\geq 6$, 
\[
\ex(n,K_r,\C^{odd}_{\geq 2k+1})\leq \frac{n-1}{2k-1}\binom{2k}{r}.
\]
The equality holds if and only if $n-1$ is divisible by $2k-1$ for a connected $n$-vertex graph which consists of $\frac{n-1}{2k-1}$ maximal $2$-connected blocks isomorphic to $K_{2k}$. 
\end{theorem}

It is interesting to ask for which integers and for which congruence classes the same phenomenon still holds. 
In this direction, we would like to rise a modest conjecture.
\begin{conjecture}
    For an integer $k\geq 2$ and a sufficiently large $n$. Let $G$ be an $n$ vertex $C_{3\ell+1}$-free graph for every integer $\ell \geq k $. Then for every $r\geq 2$, the number of cliques of size $r$ in $G$ is at most
\[
 \frac{n-1}{3k-1}\binom{3k}{r}.
\]
The equality holds if and only if $n-1$ is divisible by $3k-1$ for a connected $n$-vertex graph which consists of $\frac{n-1}{3k-1}$ maximal $2$-connected blocks isomorphic to $K_{3k}$. 
\end{conjecture}
Moreover, we would like to propose the following conjecture. 
Let $p$ be a prime number and $\C_{\geq p}^{prime}$ be the family of all cycles of length at least $p$.
\begin{conjecture}
For integers $n$, $r$ and a prime $p$ satisfying $r<p$, we have  
\[
 ex(n,K_r,\C_{\geq p}^{prime})\leq \frac{n-1}{p-2}\binom{p-1}{r}.
\]
The equality holds if and only if $n-1$ is divisible by $p-2$, for a connected $n$-vertex graph which consists of $\frac{n-1}{p-2}$ maximal $2$-connected blocks isomorphic to $K_{p-1}$. 
\end{conjecture}

\section{Proof of the main result}
\begin{proof}
We prove Theorem~\ref{thm:K_r_odd-free} by induction on the number of vertices of the graph. 
The base cases for $n\leq 2k$ are trivial. Let $G$ be a graph on $n$ vertices where $n>2k$.  We assume that every $\C^{odd}_{\geq 2k+1}$-free graph on $m$ vertices, for some $m<n$, contains at most $\frac{m-1}{2k-1}\binom{2k}{r}$ copies of $K_r$ and the equality is achieved for the class of graphs described in the statement of the theorem.

If $\delta(G)$, the minimum degree of $G$, is at most $k+2$ then we are done by the induction hypothesis
\[
K_r(G)\leq K_r(G[V(G) \setminus \{v\}])+\binom{k+2}{r-1}\leq \frac{n-2}{2k-1}\binom{2k}{r}+\binom{k+2}{r-1}< \frac{n-1}{2k-1}\binom{2k}{r}, 
\]
since $k\geq 6$.
From here, we may assume $G$ is a graph with $\delta(G)>k+2$, and each edge of $G$ is in a copy of $K_r$. 

Let $v_1v_2v_3\dots v_m$ be a longest path of $G$ such that $v_1$ is adjacent to $v_t$ and $t$ is the maximum possible among the longest paths.

\begin{claim}\label{claim:Cut_vertex}
If $t\leq 2k$, then $v_t$ is a cut vertex. 
\end{claim}
\begin{proof}
Consider a family $\P$ contains all longest paths of $G$ on the vertex set $\{v_1,v_2,v_3,\dots,v_m\}$ such that $v_tv_{t+1}\dots v_m$ is a sub-path of them. Let  $T_1$ be  the set of terminal vertices, excluding $v_m$, of paths from $\P$.  If $T_1=\{v_1,v_2,\dots,v_{t-1}\}$, then by the maximality of $t$, the vertex $v_t$ is a cut vertex. 
Let the path $u_1u_2\cdots u_{t-1}u_tv_{t+1}\cdots v_m$ be a path from $\P$ such that $u_r\in T_1$ and $u_{r+1}\notin T_1$  minimizing $r$. 
Note that $u_t=v_t$ and $\{v_1,v_2,\dots,v_t\}=\{u_1,u_2,\dots,u_t\}$. 

By the minimality of $r$, the vertex $u_1$ is not adjacent to two consecutive vertices from $\{u_{r+1},u_{r+2},\dots,v_t\}$. Since if $v_1$ is adjacent to $u_i$ and $u_{i+1}$ for  $r+1\leq i \leq t$, then 
\[
u_2u_3\cdots u_ru_{r+1}\cdots u_iu_1u_{i+1}\cdots u_t v_{t+1}\cdots v_m 
\]
forms a path from  $\P$, contradicting the minimality of $r$. 

Note that $u_1$ is not adjacent to vertices $u_{r+2},u_{r+3},\dots ,u_{2r+1}$. 
Otherwise, suppose $u_1$ is adjacent to $u_i$ for $r+2\leq i\leq 2r+1$, then  we get a path from $\P$
\[
u_{i-1}u_{i-2}\dots u_{r+1}u_{r}\dots u_1u_{i}u_{i+1}\dots u_t v_{t+1} \dots v_{m},
\]
contradicting the minimality of $r$ since  $u_{r+1}\notin T_1$.

Finally we have $d(u_1)\leq (r-1)+\frac{t-(2r+1)+1}{2}=\frac{t-2}{2}<k$, a contradiction. \end{proof}

\begin{claim}\label{claim:t<2k}
We have $t\leq 2k$.
\end{claim}

\begin{proof}
    Suppose otherwise, $t>2k$. Let $C$ denote  the cycle $v_1v_2\dots v_tv_1$.  Since $G$ is $\C^{odd}_{\geq 2k+1}$-free, $t$ is even and vertices $v_{t-3}$ and $v_{t-2}$ have no common neighbors in $G[V(G)\setminus V(C)]$. On the other hand, since every edge is in a $K_r$ there is a vertex $v_l$ incident with both $v_{t-3}$ and $v_{t-2}$.  Let us denote $c:=t-2k$. Since $G$ is $\C^{odd}_{\geq 2k+1}$-free, $c-1\leq l\leq 2k-4$ holds. We take indices modulo $t$ in this claim. 

    There is no $i$ such that $v_i$ is adjacent with $v_{t-1}$ and $v_{i+1}$ is adjacent with $v_{1}$. Since otherwise, the following cycle is odd with a length greater than $2k$,
    \[
        v_1v_2\dots v_iv_{t-1}v_{t-2} \dots v_{i+1}v_{1}.
    \]
    Note that, since $v_{t-1},v_1 \in T_1$ and from maximality of $t$, we have $N(v_1),N(v_{t-1})\subseteq V(C)$. Therefore we have $N(v_1)\cap N^+(v_{t-1})=\emptyset$. Where $N^+(v_{t-1})$ denotes the following set $\{v_{i+1}:v_i \in N^+(v_{t-1})\}$.

    There is no such $i$ such that both $l<i< l+c-2$ and $v_1v_i$ is an edge of $G$ hold. Since otherwise, one of the following cycles is an odd cycle longer than $2k$,
    \[
v_1v_2\dots v_l v_{t-2}v_{t-3}\dots v_{i}v_1 \mbox{~or~}
v_1v_2\dots v_l v_{t-3}v_{t-4}\dots v_{i}v_1.
    \]
 Similarly, there is no $i$ such that $l<i<l+c$ and $v_{t-1}v_i$ is an edge of $G$. Finally, we have 
 \[
N(v_1)\cap \{v_{l+2},\dots,v_{l+c-3}\}=\emptyset.
 \]
 Also, we have
  \[
N(v_1)\cap N^+(v_{t-1})=\emptyset
 \]
 and 
 \[
N^+(v_{t-1})\cap \{v_{l+2},\dots,v_{l+c-3}\}=\emptyset.
 \]
 Hence, we get
\[
\abs{N(v_1)}\leq \abs{\{v_1,v_2,\dots,v_{2k+c}\}\setminus N^+(v_{t-1})\setminus  \{v_{l+2},\dots,v_{l+c-3}\}}\leq 2k+c-(k+3)-(c-5)<\delta(G),  
\]
a contradiction.
\end{proof}

From Claims~\ref{claim:Cut_vertex} and \ref{claim:t<2k}, $G$ contains a $2$-connected block of size at most $2k$. By the induction hypothesis,  we see that graph $G$ contains  at most $\frac{n-1}{2k-1}\binom{2k}{r}$ edges. The equality is achieved if and only if every maximal $2$-connected component of $G$ is isomorphic to $K_{2k}$ and the number of the maximal $2$-connected components is  $\frac{n-1}{2k-1}$. 
\end{proof}

\section{Acknowledgements}
 The research of Gy\H{o}ri is partially supported by the National Research, Development and Innovation Office -- NKFIH, grant K116769, K132696 and SNN117879.  
 The research of Salia was supported by the Institute for Basic Science (IBS-R029-C4).
The research of Zhu was supported by program B for an outstanding Ph.D. candidate at Nanjing University.

\bibliography{Referances.bib}

  \textit{E-mail addresses:} \\
  E.~Gy\H{o}ri: \texttt{gyori.ervin@renyi.hu}\\
  Z.~He: \texttt{hz18@mails.tsinghua.edu.cn}\\
  J.~Lv: \texttt{lvzq19@mails.tsinghua.edu.cn}\\
  N.~Salia: \texttt{salianika@gmail.com}\\
  C.~Xiao: \texttt{xiaochuanqi@outlook.com}\\
  X.~Zhu: \texttt{ zhuxt@smail.nju.edu.cn}

\end{document}